\documentclass[11pt]{article}
\usepackage[a4paper]{geometry}

\usepackage{amsmath,amssymb,amsthm}
\usepackage{graphicx}

\newtheorem{theorem}{Theorem}
\newtheorem{lemma}{Lemma}

\newtheorem{claim}{Claim}

\newcommand{\cartprod}{\hskip1pt \Box \hskip1pt}
\newcommand{\IND}{\mathop{\alpha}}

\newcommand{\dd}[1]{\textit{#1}}  

\begin{document}

\begin{center} \Large
\textbf{Graphs in which all maximal bipartite subgraphs\\[2mm] have the same order}
\end{center}\smallskip

\begin{center}
\textsuperscript{a}Wayne Goddard
\qquad
\textsuperscript{b}Kirsti Kuenzel
\qquad
\textsuperscript{a}Eileen Melville
\end{center}

\begin{center}
\textsuperscript{a}School of Mathematical and Statistical Sciences, 
Clemson University, Clemson, SC\\[2mm]
\textsuperscript{b}Department of Mathematics, Trinity College, Hartford, CT
\end{center}

\begin{abstract}
Motivated by the concept of well-covered graphs, we define a graph to be 
well-bicovered if every vertex-maximal bipartite subgraph has the same order
(which we call the bipartite number).
We first give examples of them, 
compare them with well-covered graphs, and characterize those with small
or large bipartite number. We then consider graph operations
including the union, join, and lexicographic and cartesian products. Thereafter we
consider simplicial vertices and $3$-colored graphs where every vertex is in triangle, 
and conclude by characterizing
the maximal outerplanar graphs that are well-bicovered.
\end{abstract}

\section{Introduction}

Plummer \cite{Plummer-1970} defined a graph to be \dd{well-covered} if every maximal 
independent set is also maximum. That is, a graph is well-covered if every maximal 
independent set has the same cardinality, namely the independence number~$\IND(G)$. 
Much has been written about these graphs. For example, Ravindra~\cite{Ravindra-1977} 
characterized well-covered bipartite graphs, Campbell, Ellingham, and 
Royle~\cite{CER-1993} characterized well-covered cubic graphs, and Finbow, Hartnell, 
and Nowakowski~\cite{FHN-1993} characterized well-covered graphs of girth $5$ or 
more.

Motivated by this idea, we define a graph to be \dd{well-bicovered} if every 
vertex-maximal bipartite subgraph has the same order. Equivalently, one can define the 
\dd{bipartite number} of a graph $G$, denoted $b(G)$, as the maximum cardinality of a 
bipartite induced subgraph in $G$. (We will henceforth just assume that subgraph 
means induced subgraph.)
Then, being well-bicovered means all maximal bipartite subgraphs have cardinality 
$b(G)$. The problem of finding a maximum bipartite subgraph is well-studied. For 
instance, Zhu~\cite{Zhu-2009} showed that any triangle-free subcubic graph $G$ with 
order $n$ has $b(G) \ge \frac{5}{7}n$, and the Four Color Theorem shows that $b(G)\ge 
n/2$ for any planar graph~$G$.

In this paper, we introduce and study well-bicovered graphs. We give examples and 
compare with well-covered graphs, and characterize well-bicovered graphs with small 
or large bipartite numbers. We then consider their relationship to graph operations 
including the union, join, and lexicographic and cartesian products. Thereafter we 
consider $3$-colorable graphs and simplicial vertices, and conclude by characterizing 
the maximal outerplanar graphs that are well-bicovered.

\subsection{Definitions and terminology}

Let $G = (V(G), E(G))$ be a simple, finite graph. The open neighborhood of a vertex 
$v \in V(G)$ is $N(v) = \{ \, x \in V(G) : xv \in E(G) \,\}$. The degree of $v \in V(G)$ is 
$\deg(v) = |N(v)|$, and the maximum degree of $G$ is denoted $\Delta(G)$. Given a set 
$X \subseteq V(G)$, we let $G[X]$ represent the subgraph induced by $X$. If 
$\deg_G(x) = 1$, we refer to $x$ as a leaf in $G$, and the edge incident with $x$ as 
a pendant edge. In general, a bridge is an edge whose removal increases the number of 
components.


\section{Examples of Well-Bicovered Graphs}

In this section, we construct examples of well-bicovered graphs and study 
well-bicovered graphs whose maximal bipartite subgraphs have a given cardinality.

Trivially, every bipartite graph is well-bicovered. So are the complete graphs and 
the cycles. Further, bridges are irrelevant, as adding or removing a bridge does not 
alter the property. In particular, note for example that adding a pendant edge to any 
well-bicovered graph results in a well-bicovered graph.

\subsection{The relationship to well-covered graphs}

While the concept of being well-bicovered was motivated by the concept of being 
well-covered, the two properties are distinct. In particular, neither property 
implies the other. For example, the path $P_3$ is well-bicovered but not 
well-covered. On the other hand, the graph $F$, obtained from $K_4 - e$ and adding a 
pendant edge to a vertex of degree $2$, is well-covered but not well-bicovered. The 
house graph $H$ ($C_5$ plus a chord) is both. See Figure~\ref{fig:fishHouse}.

\newpage

\begin{figure}[h!]
\begin{center}
\begin{tabular}{c}
\includegraphics{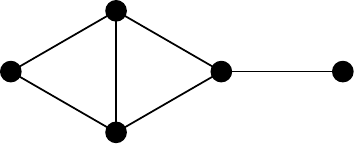}\\
$F$
\end{tabular}
 \qquad\qquad 
\begin{tabular}{c}
\includegraphics{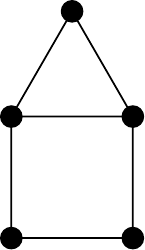}\\
$H$
\end{tabular}
\caption{Two well-covered graphs}
\label{fig:fishHouse}
\end{center}
\end{figure}  

The house graph and complete graph
have the property that their bipartite number is twice their independence number.
But there are also examples where this not the case. Two
such graphs are shown in Figure~\ref{fig:coronaRocket}.

\begin{figure}[h!]
\begin{center}
\begin{tabular}{c}
\includegraphics{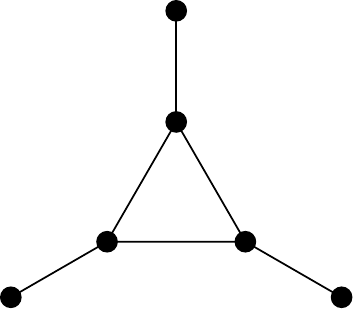}
\end{tabular}
\qquad\qquad
\begin{tabular}{c}
\includegraphics{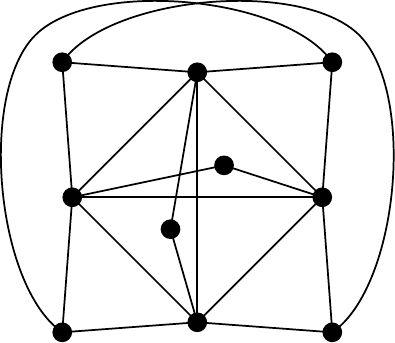}
\end{tabular}
\caption{Two well-covered and well-bicovered graphs $G$ with $b(G) < 2\IND(G)$}
\label{fig:coronaRocket}
\end{center}
\end{figure}

Though not equivalent to being well-bicovered, there is another ``bipartite subgraph'' property
that is more closely related to being well-covered, that we mention in passing. 
Let us define the ``\dd{weight}'' of a subgraph
as the sum of twice the number of isolated vertices plus the number
of nonisolated vertices. Then being well-covered implies
that every maximal bipartite subgraph has the same weight:

\begin{lemma} \label{l:weight}
If graph $G$ is well-covered, then the weight of every maximal 
bipartite subgraph is the same.
\end{lemma}

\begin{proof}
Consider a maximal bipartite subgraph $B$. Let $X$ denote
the isolates of $B$ and let 
$(Y_1,Y_2)$ denote the bipartition of~$V(B)-X$.
The maximality condition means that adding to $B$ any other vertex $v$ produces
an odd cycle. This requires that vertex $v$ be adjacent to 
both a vertex of $Y_1$ and $Y_2$. Further, every vertex of $Y_1$ has
a neighbor in~$Y_2$ and vice versa, by the definition of $Y$. 
Thus, both $X \cup Y_1$ and 
$X\cup Y_2$ are maximal independent sets in $G$: that is 
$2|X|+|V(B)-X| = 2 \IND(G)$.
\end{proof}


\subsection{Classifying graphs based on their bipartite number}

First, we classify well-bicovered graphs with small bipartite number. 

\begin{lemma}
A connected graph $G$ is well-bicovered with bipartite number $2$ if and only if 
$G=K_n$ for $n \ge 2$.
\end{lemma}

\begin{proof}
It is clear that if $G= K_n$ for $n \ge 2$, then $G$ is well-bicovered with bipartite 
number $2$. On the other hand, if $G$ is connected and not complete, then it contains 
an induced $P_3$, and so $b(G)\ge 3$.
\end{proof}

\begin{lemma} 
A connected graph $G$ is well-bicovered with bipartite number $3$ if and only if $G$ is 
obtained by taking a nontrivial complete graph $K_n$ and attaching a pendant edge to 
one vertex of $K_n$.
\end{lemma}

\begin{proof}
First, note that if $G$ is obtained by taking a nontrivial complete graph $K_n$ and 
attaching a pendant edge to one vertex of $K_n$, then $G$ is well-bipartite with 
bipartite number $3$. Conversely, suppose $G$ is well-bicovered with bipartite number 
$3$. If $G$ has order $3$, then $G = P_3$ and we are done. So we may assume that $G$ 
has order at least $4$.

Since $G$ is connected and not complete, there exists an induced $P_3$ with central 
vertex~$v$; this must be a maximal bipartite subgraph in $G$. Thus, every vertex of $G$ 
is adjacent to $v$. Suppose there exists an induced $P_3$ in $G-v$ with central 
vertex~$w$. As before, this implies that every vertex of $G$ is adjacent to $w$. However, 
$G[\{v, w\}]$ is then a maximal bipartite subgraph of $G$, which is a contradiction. It 
follows that $G-v$ is a disjoint union of cliques and we may write $G-v = K_{n_1} \cup 
\cdots \cup K_{n_j}$. We can create a maximal bipartite subgraph $H$ of $G$ by choosing 
one vertex from each $K_{n_i} = K_1$ and two vertices from each $K_{n_i}$ where $n_i 
\ge 2$. Since $|V(H)| = 3$, it follows that $G - v = K_1 \cup K_n$ where $n \ge 2$.
\end{proof}

Next, we consider the other end of the spectrum and classify well-bicovered graphs~$G$ 
with bipartite number $|V(G)| - 1$.

\begin{lemma}
A graph $G$ is well-bicovered with bipartite number $|V(G)|-1$ if and only if there 
exists an odd cycle $C$ in $G$ such that every odd cycle of $G$ contains every vertex 
on~$C$.
\end{lemma}

\begin{proof}
Suppose first that $G$ is well-bicovered of bipartite number $|V(G)| - 1$. Thus, $G$ is 
not bipartite. Let $C$ be a shortest odd cycle 
in~$G$, and let $v$ be any vertex on $C$. Note that there exists a maximal bipartite 
subgraph~$H$ of $G$ that contains $v$. If $v$ is not on every odd cycle in $G$, then 
the cardinality of $H$ is at most $|V(G)| - 2$. Thus, $v$ must lie on every odd cycle 
in $G$.

Conversely, suppose that $G$ contains an odd cycle $C$ in $G$ such that every odd cycle 
of $G$ contains every vertex on $C$. (Necessarily $C$ is chordless.) Let $H$ be a 
maximal bipartite subgraph of $G$. We know that $H$ must contain some vertex $v$ on 
$C$. Since $G - v$ is bipartite, it follows that $|V(H)| = |V(G)| - 1$. 
\end{proof}


\section{Graph Operations}

In this section we consider how the property of being well-bicovered relates 
to several graph operations including disjoint union (and related ``gluing'' operations), 
join, lexicographic product, and cartesian product.

\subsection{Union}

The question for disjoint union of graphs is trivial. The disjoint union is 
well-bicovered if and only if each component is well-bicovered.
Indeed, we observed earlier that bridges are irrelevant, and so one can take the 
disjoint union and add a bridge. 

But consider instead taking two disjoint well-bicovered graphs $G$ and $H$ and 
identifying a vertex $g$ of $G$ with a vertex $h$ of $H$ to form vertex $v$. The result 
need not be well-bicovered: consider for example $G=H=K_3$. Indeed, the result is 
guaranteed to be not well-bicovered unless for at least one of the graphs, the 
identified vertex is in every maximal bipartite subgraph. 
This holds, for example, when one of $G$ or $H$ is bipartite.
This idea is generalized slightly in the following operation:

\begin{lemma}
Let $G$ be a well-bicovered graph with adjacent vertices $u$ and $v$ 
and let $H$ be a bipartite graph with adjacent vertices $u'$ and $v'$.
Let $F$ be the graph formed from their disjoint union by identifying $u$ with
$u'$ and $v$ with $v'$. Then $F$ is well-bicovered. 
\end{lemma}
\begin{proof}
Any chordless odd cycle of $F$ has all its vertices in $G$ (since
if it uses new vertices in $H$ then $uv$ is a chord). Thus 
any maximal bipartite subgraph of $G$ can be extended to 
one of $F$ by adding all vertices of $H-\{u',v'\}$. Conversely,
any maximal bipartite subgraph of $F$ contains all of $H-\{u',v'\}$, 
and removal thereof yields a maximal bipartite subgraph of $G$.
\end{proof}

A simple example of the above is the case that $H$ is an even cycle.
One can also glue on odd cycles and general well-bicovered
graphs under some circumstances.

\begin{lemma}
Let $G$ be a well-covered graph with clique $C$ 
and let $H$ be a well-bicovered graph with clique $D$, where $|C|=|D|=k$.
Then the graph $F$ formed from their disjoint union by adding $k$ disjoint paths 
between $C$ and $D$ such that all the added
paths have the same parity, is well-bicovered. 
\end{lemma}
\begin{proof}
Any chordless cycle of $F$ containing a vertex of 
both $G$ and $H$ necessarily consists of two 
of the added paths and the edges joining their end points, and thus has 
even length. It follows that every
chordless odd cycle is contained entirely within either $G$ or $H$.
Thus, every maximal bipartite subgraph of $F$ consists
of a maximal bipartite subgraph of $G$, a maximal bipartite subgraph
of $H$, and all the interior vertices of the added paths. 
\end{proof}

We consider for a moment girth.
Finbow et al.~\cite{FHN-1993} classified all connected well-covered graphs of girth at 
least $5$. But there does not appear an 
easy characterization of well-bicovered graphs of large girth.
For example, all cycles are well-bicovered
and both the above lemmas can be used to grow a well-bicovered graph
while preserving the girth. One can even grow the girth under some circumstances
using the following lemma:

\begin{lemma}
Let $G$ be a well-bicovered graph with edge $e$ incident with a vertex 
$y$ of degree~$2$. Then the graph $G'$ obtained from $G$ by 
replacing the edge $e$ by a path of length three is well-bicovered.
\end{lemma}
\begin{proof}
Say the added path has interior vertices $u$ and $v$.
Every maximal bipartite 
subgraph $B$ of $G$ can be augmented with $u$ and $v$ to 
be a maximal bipartite
subgraph of~$G'$. Conversely, if $B'$ is a maximal bipartite
subgraph of~$G'$, then it must contains at least two of $u,v,y$;
form $B$ by removing $u,v$ if it contains both, or the two of the triple it
does contain. The resultant $B$ is a
 maximal bipartite subgraph of $G$.
\end{proof}

Of course, one can replace three by any odd number in the above
lemma, or equivalently, iterate use of the lemma.

\subsection{Join}

We consider the join next.

\begin{theorem}
The join of two nonempty graphs $G$ and $H$ is well-bicovered
if and only if each is both well-covered and well-bicovered, and 
$b(G) = b(H) = 2\IND(G) = 2\IND(H) $.
\end{theorem}

\begin{proof}
Let $B$ be any maximal bipartite subgraph of the join. If $B$ contains vertices
from both $G$ and $H$, then it must consist of an independent set
from each graph. Indeed, its vertex set must be the 
union of a maximal independent set from each graph. Since this cardinality
is constant, we need both $G$ and $H$ to be well-covered.

On the other hand, if $B$ contains vertices only from $G$, then for
its cardinality to be constant, it must be that $G$ is well-bicovered.
(And note that since $G$ has at least one edge, there do exist such~$B$.)
We get a similar result if $B$ contains only vertices from $H$. Thus
it is necessary that $b(G) = b(H) = \IND(G) + \IND(H)$. But since
$b(G)\le 2\IND(G)$, this forces $\IND(G)=\IND(H)$, and thus the above
conditions are necessary.

Finally, it easy to see that the conditions
imply that all bipartite subgraphs of the join have the same cardinality.
\end{proof}

For the case that one of the graphs is empty, one gets a similar result
with a similar proof:

\begin{theorem}
The join of $rK_1$ and $H$ is well-bicovered
if and only if $H$ is both well-covered and well-bicovered, and 
$b(H) = r + \IND(H) $.
\end{theorem}

\subsection{Well-bicovered graphs with large cliques}

We next ask what operations can be applied to a complete graph to create other 
well-bicovered graphs. 

\begin{lemma} \label{l:triangleOnEdge}
Let $G$ be the graph obtained by taking a nontrivial complete graph $K_n$ and, for each 
edge $e \in E(K_n)$, adding a vertex $v_e$ that is adjacent to both vertices incident 
to~$e$. Then $G$ is well-bicovered.
\end{lemma}

\begin{proof}
Let $B$ be a maximal bipartite subgraph of $G$. Note that $|V(B) \cap V(K_n)| \le 2$. 
If $B$ contains no vertices of the clique $K_n$, then $V(B) = \{v_e: e \in E(K_n)\}$. 
However, this is not maximal, as the graph induced by $\{v_e: e \in E(K_n)\} \cup 
\{v\}$ for any $v \in V(K_n)$ is also bipartite in $G$. If $B$ contains only one vertex 
from $K_n$, say $w$, then $V(B) = \{w\} \cup \{v_e: e \in E(K_n)\}$. If $B$ contains 
two vertices from $K_n$, say $u$ and $w$, then $V(B) = \{u, w\} \cup \{v_e: e \in 
E(K_n) - \{uw\}\}$. In every case, $|V(B)| = n+1$.
\end{proof}

As an example of Lemma~\ref{l:triangleOnEdge}, 
if we start with $K_3$ then we get the Haj\'{o}s graph or 3-sun, shown here.

\begin{figure}[h!]
\begin{center}
\includegraphics{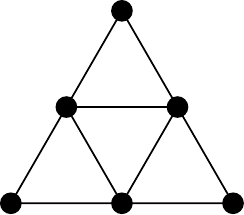}
\end{center}
\caption{A well-bicovered graph}
\end{figure}

\begin{lemma}
Let $G$ be the graph obtained by taking a disjoint union of nontrivial complete graphs $K_{n_1} \cup \cdots \cup K_{n_{\ell}}$ 
and adding edges between the cliques so that the added edges form a matching. Then $G$ is well-bicovered. 
\end{lemma}

\begin{proof}
Let $M$ represent the matching added and let $S$ be a subset of $V(G)$ formed by taking 
two vertices from $K_{n_i}$ for each $1 \le i \le \ell$. Then in $G[S]$ the edges not 
in $M$ form a matching. This means that every cycle in $G[S]$ alternates between $M$ 
and non-$M$ edges and so has even length. It follows that $G[S]$ is bipartite. On the 
other hand, no bipartite 
subgraph can have three vertices from any of the cliques. It follows that every maximal 
bipartite subgraph has exactly two vertices from each clique. The result follows.
\end{proof}

\subsection{Lexicographic product}

Recall that the lexicographic product (or composition) of graphs $G$ and $H$, denoted 
$G\circ H$, is the graph with $V(G\circ H)= V(G) \times V(H)$ whereby 
$(u, v)$ and $(x, y)$ are adjacent if $ux \in E(G)$,  or $u=x$ and $vy \in E(H)$. 

Topp and Volkmann \cite{TV-1992} proved that the lexicographic product 
$G\circ H$ of two nonempty graphs $G$ and $H$ is well-covered 
if and only if $G$ and $H$ are well-covered graphs. We now 
determine when the lexicographic product $G\circ H$ is well-bicovered. 
If $H$ is empty, the result is immediate:

\begin{lemma}
$G\circ mK_1$  is well-bicovered if and only if $G$ is well-bicovered.
\end{lemma}

Now we consider the case when both $G$ and $H$ are nonempty. 
In the following, given a vertex $x \in V(G)$, we refer to the subgraph of $G\circ H$ 
induced by $\{(x, v):v \in V(H)\}$ as the  $H^x$-fiber. 

\begin{theorem}
Let $G$ and $H$ be nonempty graphs. 
Then $G\circ H$ is well-bicovered if and only if\\
(i) $G$ is well-covered; and\\
(ii) $H$ is both well-covered and well-bicovered, and also $b(H)=2 \IND(H)$.
\end{theorem}

\begin{proof}
Assume first that $G\circ H$ is well-bicovered.
Define a \dd{good} pair $(X,Y)$ as 
disjoint subsets $X$ and $Y$ of $V(G)$ such that 
$X$ is independent, there is no edge between $X$ and $Y$, 
and $Y$ induces a bipartite subgraph without isolates.
We say that a good pair is \dd{maximal} if there is no other good pair
$(X',Y')$ such that $X\subseteq X'$ and $Y\subseteq Y'$.
Note that a maximal good pair has the following property:
If $z$ is any vertex of $V(G)-(X\cup Y)$, and $z$ is not in $N(X)$, 
then since $z$ cannot be added to $X$ it must have a neighbor in $Y$,
and since $z$ cannot be added to $Y$, it must create an odd cycle with $Y$. 
By definition, any good pair can be extended to a maximal good pair.

Now, given a maximal good pair $P=(X,Y)$, one can construct a subset 
$B_P$ of $V(G\circ H)$ as follows. 
For every $H^x$-fiber where $x \in X$, 
take a maximal bipartite subgraph of $H$. 
For every $H^y$-fiber where $y \in Y$, take 
a maximal independent set of $H$. The resultant set $B_P$ is clearly bipartite.
Further:

\begin{claim}
The subgraph induced by the set $B_P$ is maximal bipartite. 
\end{claim}
\begin{proof}
Consider adding another vertex $v$ to $B_P$; say from the $H^w$-fiber.
If $w \in X$, then vertex $v$ creates an odd cycle with $B_P$, 
since we already took a maximal bipartite subgraph of such $H^w$.
If $w \in N(X)$, say adjacent to $x\in X$, 
then vertex $v$ creates a triangle with the vertices of $B_P$ in $H^x$,
since any maximal bipartite subgraph of $H^x$ has at least one edge.
If $w\in Y$, say adjacent to $y\in Y$, then vertex $v$ creates a triangle 
with a vertex of $B_P$ in $H^w$ and in $H^{y}$. 
Finally, by the maximality of the good pair, if $w\in V(G)-N[X]-Y$,
then vertex $v$ creates an odd cycle with $B_P$, since $w$ creates
an odd cycle with $Y$.
\end{proof}

Consider a good pair $P_1$ with $X$ a maximal
independent set of $G$ and $Y$ empty; necessarily $P_1$ is a maximal good pair.
Since all resultant sets $B_{P_1}$ must have the same size,
it follows that $H$ must be well-bicovered, and
$G$ must be well-covered. Further, 
every maximal bipartite subgraph of $G\circ H$ must
have size
\[
  |B_{P_1}| = \IND(G)\, b(H) .
\]

Consider a maximal good pair $P_2$ where $Y$ is nonempty and $X\cup Y$ is maximal 
bipartite. (This exists since $G$ has at least one edge: start with such an edge, 
extend to a maximal bipartite subgraph, and then partition into isolates and 
nonisolates.) Since every resultant set $B_{P_2}$ must have the same size, it follows 
that $H$ must be well-covered. Further, the resultant $B_{P_2}$ must have size
\[
    |B_{P_2}| = |X| b(H) + |Y| \IND(H) .
\]
By Lemma~\ref{l:weight}, it holds that $2|X|+ |Y| = 2 \IND(G)$. 
Thus the condition $|B_{P_1}|=|B_{P_2}|$ is equivalent to 
$|Y|\IND (H) = 2 |Y| b(H)$. Since $|Y|\neq 0$, it follows that 
(it is necessary
and sufficient that) $\IND(H) = 2 b(H)$.
That is, we have shown that conditions (i) and~(ii) are necessary.

Conversely, assume conditions (i) and (ii) hold. Let $B'$ be a maximal bipartite 
subgraph in the composition. Let $P$ be the subset of $V(G)$ in the projection of $B'$ 
onto~$G$. Say $X$ is the isolated vertices in the subgraph induced by $P$, and $Y$ the 
non-isolates. By the maximality of $B'$, for every $H^x$-fiber where $x \in X$, the set 
$B'$ must contain a maximal bipartite subgraph of $H$, and for every $H^y$-fiber where 
$y \in Y$, it must contain a maximal independent set of $H$. It follows that
\[
   |B'| = |X| b(H) + |Y| \IND(H) .
\]
Furthermore, the maximality of $B'$ means that $(X,Y)$ is a maximal
good pair and as above, by Lemma~\ref{l:weight} we get that $|Y| = 2(\IND(G) - |X|)$.
It follows that $B'$ has size  $\IND(G) b(H)$. Since this is
true for all maximal bipartite subgraphs of $G\circ H$, 
we have that $G\circ H$ is well-bipartite.
\end{proof}

\subsection{Cartesian product}

Recall that the Cartesian product of graphs $G$ and $H$, denoted $G\cartprod H$, is the 
graph with $V(G\cartprod H) = \{(u,v): u \in V(G) \text{ and } v \in V(H)\}$ and 
$(u,v)(x,y) \in E(G\cartprod H)$ if either $u = x$ and $vy \in E(H)$ or $ux \in E(G)$ 
and $v = y$. An obvious guess is that the product of two well-bicovered graphs is 
well-bicovered. However, this is far from true. Indeed, it does not hold even if one 
graph is $K_2$. For example, the house graph $H$ has bipartite number $4$, and so 
$b(K_2 \cartprod H) = 8$. But in the prism shown in Figure~\ref{f:prism}, the dark vertices represent a 
maximal bipartite subgraph order $7$:

\begin{figure}[h!]
\begin{center}
\includegraphics{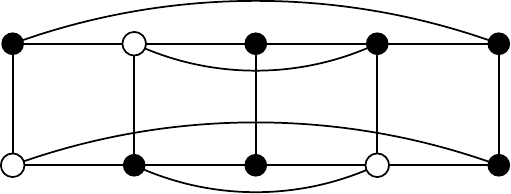}
\end{center}
\caption{A Cartesian product that is not well-bicovered}
\label{f:prism}
\end{figure}

One can at least observe that if $H$ is bipartite, then a necessary condition for
$G\cartprod H$ to be well-bicovered is that $G$ is well-bicovered. For, one can
build a maximal bipartite subgraph of the product by starting with any maximal
bipartite subgraph $B$ of $G$ and taking these vertices in all copies of $G$.
In order for the result to always be the same size, it is necessary that all
the $B$ have the same cardinality.

In \cite{HR-2013}, Hartnell and Rall showed that if $G\cartprod H$ is well-covered, 
then 
one of $G$ or $H$ must be well-covered. We do not know the answer to the 
analogous question: namely, 
if $G\cartprod H$ is well-bicovered, then must one of $G$ or $H$ be 
well-bicovered?

\section{$3$-Colorable Graphs}

As we mentioned earlier, 
Ravindra~\cite{Ravindra-1977} characterized the well-covered bipartite graphs.
So one might hope to classify well-bicovered $3$-colorable graphs, 
but this seems challenging. We note that in the 
case of well-covered bipartite graphs, one can trivially assume that every vertex is
in a $K_2$. So perhaps one can characterize well-bicovered $3$-colorable graphs 
where every vertex is in a triangle. We present here some partial results. We then
use these to characterize well-bicovered maximal outerplanar graphs.

\subsection{Triangles and simplicial vertices}

\begin{lemma} \label{l:3chrom}
Let $G$ be a $3$-colorable well-bicovered graph such that every vertex is in a triangle.
Then each color class $V_i$ in every proper $3$-coloring of $G$ has size $b(G)/2$. Further,
if every edge of $G$ is in a triangle, then each subgraph $G-V_i$ is well-covered. 
\end{lemma}
\begin{proof}
Let $T_i = G- V_i$ for each $1\le i \le 3$. Because the subgraph $T_i$ is induced by two color classes,
it is bipartite. If $w$ is any other vertex, then it has color $i$;
by construction $T_i$ contains all the neighbors of $w$ and so
adding $w$ to $T_i$ would create a triangle with~$T_i$. That is,
 $T_i$ is a maximal bipartite subgraph. 
Therefore, $|T_1|=|T_2|=|T_3|$.
It follows that $|V_1|=|V_2|=|V_3|$;
and indeed, each is $b(G)/2$.

Now, suppose every edge is in a triangle.
We can build a bipartite subgraph $B$ of $G$ by starting with 
the color class~$V_i$ and then adding 
a maximal independent set $S$ of~$T_i$. 
Consider some other vertex $x$ of~$T_i$. Then $x$ has a neighbor $y$ in $S$; further,
the edge $xy$ is in a triangle, say with vertex~$z$, where $z$ is in $V_i$.
That is, if we add $x$ to $B$ 
we complete a triangle. It follows that $B$ is maximal bipartite.
Since all such $B$ must have the same cardinality, we get that $T_i$ is well-covered.
\end{proof}

Recall that a \dd{simplicial vertex} is one whose neighborhood is complete.
Prisner, Topp, and Vestergaard~\cite{PTV-1996} considered simplicial vertices in well-covered graphs.
In particular, they defined a \dd{simplex} as a maximal clique containing a simplicial vertex,
and showed that if a graph is well-covered then the simplices are vertex-disjoint, and if
every vertex belongs to exactly one simplex then the graph is well-covered.

It is unclear what the exact analogue of their results should be. 
For example, Lemma~\ref{l:triangleOnEdge}
showed the well-bicovered graphs can have overlapping simplices. But here are two results
in that spirit. 

\begin{lemma}
Suppose graph $G$ has vertices $u$ and $v$ that are nonadjacent simplicial vertices 
with $N(u)\cap N(v)$ nonempty, and there exist distinct vertices $x\in N(u)$ and $y\in N(v)$ 
that are nonadajcent.
Then $G$ is not well-bicovered.
\end{lemma}
\begin{proof}
Let $c\in N(u) \cap N(v)$. Note that the condition implies that $x\notin N(v)$ and
$y\notin N(u)$.
Take the set $\{c,x,y\}$ and extend to a maximal bipartite set $B$. Necessarily, the 
set~$B$ cannot contain $u$ or $v$ nor any other vertex of $N(u) \cup N(v)$.
Now, let $B'$ be the set $(B\cup\{u,v\})-\{c\}$. Then this set is bipartite
and bigger than $B$, a contradiction.
\end{proof}

Define a \dd{bisimplex} $S$ as a maximal clique that contains a simplicial vertex
$s$ (implying it has no neighbor outside $S$) and a second vertex $t$ that has at 
most one neighbor outside~$S$.
 
\begin{lemma}  \label{l:biSimplex}
If the vertex set of graph $G$ has a partition into bisimplices,
then $G$ is well-bicovered.
\end{lemma}
\begin{proof}
Let the partition of $V(G)$ into bisimplices be $S_1, \ldots, S_m$, 
with $s_i$ the simplicial vertex of $S_i$ and $t_i$ the other relevant vertex of $S_i$.
Let $B$ be the subgraph induced by all the $s_i$ and $t_i$.
Then $B$ is bipartite: indeed, every component in $B$ is a path either of 
length $1$ (of the form $s_it_i$) or of length $3$ (of the form $s_it_it_js_j$).
Since one cannot take more than two vertices
from each bisimplex, it follows that $b(G) = 2m$.

Now consider any maximal bipartite subgraph $B'$ of $G$. Suppose $B'$ contains no vertex from
some bisimplex $S_i$. Then one can add the vertex $s_i$ to $B'$ and preserve bipartiteness. 
Suppose $B'$ contains only one vertex from $S_i$. If that vertex is not $s_i$, then one can 
just add it. If that vertex is $s_i$, then one can add $t_i$,
as it is adjacent to at most one other vertex in $B'$. In either case this contradicts
the claimed maximality of~$B'$. 
It follows that $B'$ contains two vertices from each bisimplex, and so has cardinality~$2m$.
\end{proof}

\subsection{Maximal outerplanar graphs}

Recall that graph $G$ is \dd{outerplanar} if $G$ has a planar drawing for which all 
vertices belong to the outer face of the drawing. 
We say that $G$ is \dd{maximal outerplanar}, or a \dd{MOP}, if $G$ is outerplanar and 
the addition of any edge results in a graph that is not outerplanar. 
Since every MOP is a $2$-tree and every $2$-tree is chordal, we point out that 
Prisner, Topp, and Vestergaard~\cite{PTV-1996} classified well-covered chordal graphs. 
In this section, we classify all well-bicovered MOPs. 

We construct a class of graphs $W$ referred to as \dd{whirlygigs} as follows. 
Take a MOP $M$ with $m\ge 3$ vertices and let $C$ represent the outside cycle of $M$. 
For each edge $u_iu_{i+1}$ on $C$, add a new vertex $t_i$ 
adjacent only to $u_{i}$ and $u_{i+1}$, and then add 
another vertex $s_i$ that is adjacent to only $t_i$ and $u_i$. 
Figure~\ref{fig:whirlygig} shows the whirlygig of order $15$ 
(there is only one MOP of order $5$). 
We can extend this definition to $m=2$ by considering the MOP $K_2$ to be a cycle with 
two edges; the resultant whirlygig is $P^2_6$, the square of the path.

\begin{figure}[h!]
\begin{center}
\begin{tabular}{c}
\includegraphics{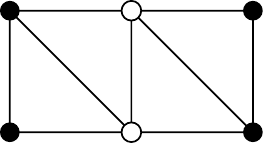} 
\end{tabular}
\qquad\qquad
\begin{tabular}{c}
\includegraphics{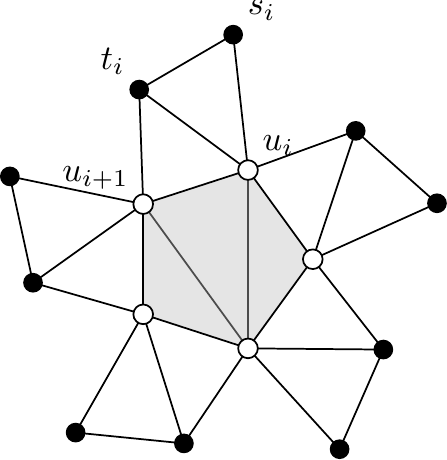}
\end{tabular}
\caption{The unique whirlygigs of orders $6$ and $15$}
\label{fig:whirlygig}
\end{center}
\end{figure}

\begin{lemma}
If $G \in W$, then $G$ is well-bicovered.
\end{lemma}
\begin{proof}
This follows from Lemma~\ref{l:biSimplex}.
Each $\{s_i,t_i,u_i\}$ is a bisimplex, and these partition the vertex set of $G$.
\end{proof}

\begin{theorem} 
A MOP $G$ is well-bicovered if and only if $G$ is $K_3$, the Haj\'{o}s graph, or a 
whirlygig. 
\end{theorem}

\begin{proof}
Let $G$ be a well-bicovered MOP. It is well-known that outerplanar graphs are $3$-colorable. 
Let $V_1, V_2$, and $V_3$ be the 
color classes of $G$. Let $T_i = G- V_i$ for each $1\le i \le 3$. 
Note that $T_i$ is a tree.
(This follows for example from the fact that $G$ is chordal and 
therefore the subgraph induced by the vertices of a cycle cannot be 
$2$-colored.) By Lemma~\ref{l:3chrom}, we know that 
the cardinality of $V_1, V_2$, and $V_3$ must be equal; say 
$|V_i| = m$ for $1\le i \le 3$. Moreover, each $T_i$ is 
well-covered. By the characterization of Ravindra~\cite{Ravindra-1977},
it follows that each $T_i$ is a corona of a tree. 

We partition $V(G)$ into three sets. Let LL represent the vertices in $G$ that are a leaf 
with respect to $T_i$ and $T_j$ for some $1 \le i < j \le 3$. Let LN represent the vertices 
that are a leaf with respect to $T_i$ and a non-leaf with respect to $T_j$, and let NN 
represent the vertices that are a non-leaf with respect to $T_i$ and $T_j$
for some $1 \le i < j \le 
3$. We know that the only MOP of order three is $K_3$. So we may assume that 
$|V(G)| \ge 6$.\smallskip

\textbf{Case 1.}
Suppose first that $G$ contains a vertex $v_1$ in LN. Without loss of generality, we may 
assume $v_1$ has color $1$, is a leaf in $T_2$ and is a non-leaf in $T_3$. Thus, $v_1$ has 
exactly one neighbor of color $3$, say $x_1$, and at least two neighbors of color $2$, say 
$w_1$ and $x_2$. 

Since in a MOP the open neighborhood of a vertex induces a path, any vertex 
must have almost equal representation of the other two colors in its neighborhood. It follows 
that $v_1$ has exactly two neighbors of color $2$, and in particular, its open neighborhood 
induces the path $w_1x_1x_2$. Further, one of $w_1$ and $x_2$ is a leaf in $T_3$, say $w_1$. 
Since $w_1$ has only one neighbor of color $1$, it must be that the edge $w_1x_1$ is an 
exterior edge. Since $v_1$ has only one neighbor of color $3$, the edge $v_1w_1$ is also an 
exterior edge. In particular, $w_1$ has degree $2$ in $G$.

\begin{center}
\includegraphics[scale=1.00]{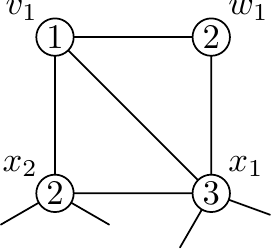}
\end{center}

Note that the above argument holds for any vertex that is in LN. We refer to $v_1w_1x_1$ as a 
\dd{leafy triangle}
as all three vertices are leaves in their respective coronas. Note too that 
any vertex in $G$ is incident with exactly one pendant edge from two different coronas $T_i$ 
and $T_j$. In particular, $x_1$ is not contained in another leafy triangle of $G$.

Next, consider vertex $x_2$. Since $w_1$ is a leaf in both $T_1$ and $T_3$, it follows that 
$x_2$ must be a non-leaf in both $T_1$ and $T_3$. Therefore, $x_2$ has a leaf-neighbor in 
$T_1$ and a leaf-neighbor in $T_3$. We claim that one of these leaf-neighbors, call it $v_2$, 
must be in LN. Indeed, if both leaf-neighbors were in LL, then they would each have degree 
$2$ in $G$ making $x_2$ a cut-vertex which cannot happen. 

Applying the above argument where 
$v_2$ plays the role of $v_1$, it follows that $v_2$ has exactly two neighbors other than 
$x_2$, say $w_2$ and $x_3$, where $x_2w_2v_2x_3$ is a path on the exterior of $G$ and $x_2$, 
$w_2$, and $v_2$ induce a leafy triangle. Continuing this same line of reasoning, we 
establish the exterior path $P = x_1w_1v_1x_2w_2v_2\dots x_kw_kv_k$ where each $x_i$ is in 
NN, each $w_i$ is in LL, and each $v_i$ is in LN. Furthermore, from above we know that $v_k$ 
has degree $3$ in $G$ and is adjacent to $x_k$ and $w_k$. 

Let $k$ be the first index where 
the third neighbor of $v_k$ is $z$ which is already on $P$. Since each $w_i$ has degree $2$ 
in $G$, and each $v_i$ is only adjacent to $w_i$, $x_i$, and $x_{i+1}$, it must be that $z = 
x_i$ for some $1\le i \le k$. If $k=2$, then $G$ is the graph depicted $P^2_6$. So we may 
assume that $k\ge 3$. If $i \ne 1$, then $x_i$ is a cut-vertex as the exterior edges of $G$ 
contain the cycle $x_iw_iv_i\dots x_kw_kv_kx_i$. It follows that all vertices of NN are on 
the cycle $x_1x_2\dots x_kx_1$. Moreover, the vertices in NN must induce a MOP, and thus  
$G$ is in fact a whirlygig.\smallskip

\textbf{Case 2.}
Next, suppose LN $= \emptyset$. Let $v$ be a vertex in LL with color $1$. Thus, $v$ has 
degree $2$ in $G$. Let $u$ be the neighbor of $v$ with color $2$ and let $w$ be the neighbor 
of $v$ with color $3$. It follows that $u$ and $w$ are adjacent in $G$. Since $w$ is not a 
leaf in $T_1$, $w$ has a neighbor, call it $x$, in LL that has degree $2$ and color $2$. Let 
$z$ be the other neighbor of $x$ which is necessarily in NN, has color $1$ and is also 
adjacent to $w$. Continuing this same line of reasoning, we deduce that the exterior edges of 
$G$ can be expressed as $v_1w_1v_2w_2\cdots v_kw_kv_1$ where each $v_i$ is in LL and each 
$w_i$ is in NN. Further, $G$ contains the cycle $C= w_1w_2\cdots w_kw_1$ as each $v_i$ has 
degree $2$ in $G$. Without loss of generality, we may assume $v_1$ has color $1$ and $w_1$ 
has color $2$. From the above argument, it follows that $v_2$ has color $3$, $w_2$ has color 
$1$ and so forth. This implies that on $C$, the order of the colors starting with the color 
of $w_1$ is $2, 1, 3, 2, 1, 3, \dots, 2, 1, 3$.

Let $H$ be the MOP induced by the vertices of NN. We claim that based on the pattern of 
colors on $C$, each vertex of degree $3$ or more in $H$ is adjacent to a vertex of degree $2$ 
in $H$. Indeed, let $uvw$ be on $C$ where $u$ has color $2$, $v$ has color $1$ and $w$ has 
color $3$. We shall assume that $v$ has degree at least $3$ in $H$ and that $v$ is adjacent 
to a vertex $z \ne u$ with color $2$. Thus, on $C$ vertex $z$ is followed by a vertex $t$ 
with color 
$1$. However, either $uz$ or $vt$ must be an edge in $H$, which is a contradiction based on 
their color assignment. Thus, either $u$ has degree $2$, or the only neighbor of $v$ with 
color $2$ is $u$. However, assuming that $u$ is adjacent to a vertex $p\ne w$ with color $3$, 
the same argument implies that $w$ has degree $2$ in $H$.

Next, we claim that the subgraph $J$ of $G$ containing all vertices of degree $2$ in $H$ 
along with all vertices in LL in $G$ is maximal bipartite. Indeed, for $v \in V(G) - V(J)$, 
$v$ is in NN and therefore contained in a triangle $vxw$ where $x$ has degree $2$ in $H$ and 
$w$ is in LL. Thus, $|V(J)| = \frac{2}{3}n = |\textrm{LL}| + x = \frac{n}{2} + x$ where $x$ 
represents the number of vertices in $H$ of degree $2$. It follows that $x = \frac{n}{6}$ and 
every third vertex on $C$ is a vertex of degree~$2$ in $H$. Recoloring if necessary, we may 
assume every vertex of degree $2$ in $H$ has color~$1$ and each vertex of $H$ with color $2$ 
or $3$ has degree at least $3$ in $H$. If $|V(H)| = 3$, then $G$ is the Haj\'{o}s graph. 
However, if $|V(H)| \ge 6$, then $H$ does not induce a MOP as each vertex of color $2$ is 
adjacent to exactly one vertex of color $1$ and therefore only two vertices of color $3$,
and so that case is impossible. 
\end{proof}

\section{Further Thoughts}

Apart from the questions mentioned in the text, there are several
natural questions yet to be resolved. For example, it would be nice
to characterize well-bicovered planar or outerplanar graphs of given girth,
or well-bicovered triangulations. Another obvious direction is to 
establish the complexity of recognizing well-bicovered graphs.


\bibliography{wellBi}
\bibliographystyle{elsart-num-sort}

\end{document}